\documentclass[10pt, reqno]{amsart}
\usepackage{euscript,amscd,amsgen,amsfonts,amssymb,latexsym}

\newcommand{\eqdef}{\stackrel{\scriptscriptstyle\rm def}{=}}

\newtheorem{theorem}{Theorem}

\newtheorem{proposition}{Proposition}

\newtheorem{lemma}{Lemma}
\newtheorem{remark}{Remark}
\newtheorem{example}{Example}

\newcommand{\beha}{\begin{enumerate}}
\newcommand{\behe}{\end{enumerate}}
\renewcommand{\epsilon}{\varepsilon}

\newcommand{\cM}{\EuScript{M}}

\newcommand{\cP}{\EuScript{P}}
\newcommand{\cN}{\EuScript{N}}
\newcommand{\cU}{\EuScript{U}}

\newcommand{\bR}{{\mathbb R}}
\newcommand{\bC}{{\mathbb C}}
\newcommand{\bZ}{{\mathbb Z}}
\newcommand{\bN}{{\mathbb N}}

\newcommand{\cK}{{\mathcal K}}

                  \def\oc{\overline \C}

\def\1{1\!\!1}

\def\and{\text{ and }}

\def\oc{\overline{\bC}}

\def\Per{{\rm Per}}
\def\Fix{{\rm Fix}}
\def\EPer{{\rm EPer}}
\def\EFix{{\rm EFix}}
 
\def\card{{\text{card}\,}}

\def\1{1\!\!1}

\renewcommand{\emptyset}{\varnothing}
\title{On the  distribution of periodic orbits
}
\author{Katrin Gelfert}\address{Department of Mathematics, Northwestern University,  2033 Sheridan Road, Evanston, IL 60208-2730, USA}\email{gelfert@pks.mpg.de} \urladdr{http://www.pks.mpg.de/~gelfert/}

\author{Christian Wolf}\address{Department of Mathematics, Wichita State University, Wichita, KS 67260, USA}\email{cwolf@math.wichita.edu}
\urladdr{http://www.math.wichita.edu/~cwolf/}

\begin{document}
\thanks{K.G. is grateful for discussions with A.~Arbieto and M.~Viana.}

\begin{abstract}
Let $f:M\to M$ be a $C^{1+\epsilon}$-map on a smooth Riemannian manifold $M$ and let $\Lambda\subset M$ be a  compact $f$-invariant locally maximal set. In this paper we obtain several results concerning the distribution of the
periodic orbits of $f|\Lambda$. These results are non-invertible and, in particular, non-uniformly hyperbolic versions of  well-known results by Bowen, Ruelle,
 and others in the case of hyperbolic diffeomorphisms.
 We show  that the topological pressure $P_{\rm top}(\varphi)$ can be computed by the values of the potential $\varphi$ on the expanding periodic orbits and also that every hyperbolic ergodic invariant  measure is well-approximated by expanding periodic orbits. Moreover, we prove
that certain equilibrium states are Bowen measures. Finally, we derive a large deviation result
for  the periodic orbits whose time averages are apart from the space average  of a given  hyperbolic invariant measure. 
\end{abstract}
\subjclass[2000]{}
\maketitle

\section{Introduction}
\subsection{Motivation}
It is a major goal in the theory of 
dynamical systems to determine the mechanisms which create deterministic chaos. One key ingredient causing chaotic behavior is the existence of a large set  of hyperbolic periodic orbits. Here the meaning of a set being large may vary depending on the circumstances. For instance, large can mean that the periodic points are dense in phase space (which is one condition in Devaney's definition of a chaotic dynamical system). In a topological setting it can mean an exponential growth rate of the number of periodic orbits. In a measure-theoretic setting it could signify that these orbits are uniformly distributed indicating maximal entropy.  
The connection between the distribution of the periodic orbits and other topological quantifiers such as entropy and topological pressure is well  understood in the case of uniformly hyperbolic
systems due to pioneer work of Bowen, Ruelle, Walters, Sinai, and others. One reason for  the success when dealing with hyperbolic systems is that the existence of Markov partitions allows one to
 model the system by a shift map over a finite alphabet. We refer to~\cite{Bow:71} for the case of hyperbolic diffeomorphisms and~\cite[Chapter~20]{KatHas:95} for a collection of general results. Although the analysis of further quantifiers such as Lyapunov exponents, fractal dimensions, or Gibbs measures  depends on smoothness properties of the conjugating map, this symbolic description nevertheless can provide a powerful tool. 

Beyond the realm of uniformly hyperbolic systems, symbolic descriptions (even by means of an infinite alphabet) may fail. For particular classes of dynamical systems a description of topological dynamical quantifiers in terms of periodic points can nevertheless be achieved. We refer to \cite{Oli:03}, \cite{OliVia:06} and \cite{OliVia:07} 
for results in this direction.
 In our previous work~\cite{GelWol:08} we treated the case of non-uniformly hyperbolic diffeomorphisms. This work was inspired by the seminal work of  Misiurewicz and Szlenk~\cite{MisSzl:80} and of Katok~\cite{Kat:80}   of finding horseshoes that can be used to gradually approximate the dynamics and dynamical quantifiers ``from inside'' (see also~\cite{KatMes:98} for related results in the case of $C^1$ circle maps). In the present paper we continue and extend this analysis and now study a class of non-uniformly hyperbolic (not-necessarily invertible) $C^{1+\epsilon}$-maps. We show that the topological pressure is determined by the values of the potential on the periodic points, establish the Bowen-measure property for certain equilibrium measures, and obtain large-deviation results for those periodic points whose time-averages differ from the space average of a given hyperbolic measure. We shall now discuss our results in more detail.
 
\subsection{Statement of the results}

 Let $f\colon M\to M$ be a $C^{1+\varepsilon}$ map on a smooth Riemannian manifold $M$, and let $\Lambda\subset M$ be a compact locally maximal invariant set. Let $\varphi\colon\Lambda\to\bR$ be a potential belonging to the class $C^f(\Lambda,\bR)$ (see Section 3 for the precise definition). We show  that for
any sufficiently small  $\alpha>0$ the topological pressure $P_{\rm top}(\varphi)$ is given by
  \begin{equation}\label{erg1}
  P_{\rm top}(\varphi) =   
  \lim_{\ell\to\infty}\limsup_{n\to\infty}\frac{1}{n}\log\left(
  \sum_{x\in \EFix(f^n,\alpha,\ell)} \exp S_n\varphi(x)\right),
  \end{equation}
where $\EFix(f^n,\alpha,\ell)$ denotes a set of fixed points of $f^n$ which have sufficiently strong expansion properties characterized by the exponent $\alpha$ and the parameter $\ell$ (see  Theorem~\ref{Main1} in the text).  

A second focus of our investigations is on existence and construction of equilibrium states (and in particular on measures of maximal entropy). Here an invariant measure $\mu$ is said to be an \emph{equilibrium state} for a continuous potential $\varphi\colon\Lambda\to\bR$ if  
\[P_{\rm  top}(\varphi)=h_{\mu}(f)+\int_\Lambda\varphi \,d\mu.\]
For the particular case $\varphi=0$,  $\mu$ is called a measure of \emph{maximal entropy}.
Recall that (asymptotic entropy-) expansiveness of $f|\Lambda$ guarantees  the upper semi-continuity of the entropy map $\mu\mapsto h_\mu(f)$ (see \cite{Mis:73}) and hence implies the existence of equilibrium states for any continuous potential. In particular, if $f$ is a $C^\infty$-map, then the entropy map is upper semi-continuous (see \cite{New:89}). Moreover, an equilibrium state can be obtained as a limit distribution of a convex combination of measures that are uniformly distributed on each given periodic orbit.
If $f\colon \Lambda\to\Lambda$ is an expansive homeomorphism satisfying the specification property,
this construction yields a unique equilibrium state for any H\"older continuous potential $\varphi$ (see~\cite[Chapter~20]{KatHas:95}).  Beyond such a setting the mere existence of equilibrium states is largely unknown.  

We show that any ergodic invariant measure $\mu$ with positive entropy and positive Lyapunov exponents is well-approximated by weighted distributions on expanding periodic orbits (see Theorem~\ref{thwell}). In particular, if $\mu$ is the unique equilibrium state of a certain potential, then it is a Bowen measure (see Theorem \ref{thmut}), that is, $\mu$ is a weak$\ast$ accumulation point of the probability measures $(\mu_n)_{n\in \bN}$ given by
\begin{equation}
\mu_n=\frac{1}{n}\sum_{i=1}^{n-1}f^i_*\nu_n, \quad\text{ where }
\nu_n=\frac{1}{A_{n_k}}\sum_{x\in \EFix(f^{n_k},\alpha,k)}
\delta_x\exp S_{n_k}\varphi(x)
\end{equation}
and $A_{n_k}$ denotes the normalizing constant.
We further discuss dynamical systems that are intrinsically ergodic, for which we can verify existence of equilibrium states of low-varying potentials. Here we are inspired by the examples in~\cite{Oli:03}.

We conclude our analysis with an investigation of the distribution of periodic orbits that are far from a given dynamical behavior. More precisely, we consider periodic orbits whose Birkhoff averages differ from the expected value of a given ergodic measure $\mu$, namely,
\[
\left\lvert \lim_{n\to\infty}\frac{1}{n}\sum_{k=0}^{n-1}\varphi(f^k(x))- \int\varphi \,d\mu\right\rvert 
		>\varepsilon.
\]
This analysis is done by means of a rate function which measures the ``distance'' of a periodic-orbit measure to the set of (generalized) equilibrium states.
One particular large deviation result (see Theorem \ref{theorem:ld}) states that for an ergodic measure $\mu$ with positive entropy and positive Lyapunov exponents  and a continuous potential $\varphi$ for $\alpha>0$ and  $\delta>0$ we have
	\begin{multline} 
	\lim_{\ell\to\infty}\limsup_{n\to\infty}\frac{1}{n}\log
	{\rm card}\left\{x\in\EFix(f^n,\alpha,\ell) \colon 
	\left\lvert \frac{1}{n}S_n\varphi(x) - \int_\Lambda\varphi \,d\mu\right\rvert \ge \delta \right\}\\
	        \le 
		\sup\left\{\widehat h_\nu(f)\colon 
	        \nu\in\cM, \left\lvert \int_\Lambda\varphi \,d\nu-\int_\Lambda\varphi \,d\mu\right\rvert
	        \ge \delta \right\},
	\end{multline}
where $\widehat h_\nu(f)$ denotes the generalized entropy of $\nu$ (which is equal to $h_\nu(f)$ if and only if the entropy map is upper semi-continuous).

This paper is organized as follows. In Section 2 we review some
basic concepts and results from smooth ergodic theory. 
Section 3 is devoted to the proof of   formula \eqref{erg1} showing that the topological pressure can be calculated in terms of the values of the potential on the periodic orbits.
 In Section 4 we analyze  how hyperbolic measures can be approximated by expanding periodic orbits. In particular, we obtain that certain  equilibrium measures are Bowen measures. Finally, in 
Section 5 we derive large deviation results concerning periodic orbits whose time averages differ from the space average of a given hyperbolic measure.

\section{Notions from smooth ergodic theory}

Let $M$ be a smooth Riemannian manifold and let
$f\colon M\to M$ be a $C^{1+\epsilon}$ map. We consider a compact locally
maximal invariant set $\Lambda\subset M$. Here \emph{locally
maximal} means that there exists an open neighborhood $U\subset M$
of $\Lambda$ such that $\Lambda=\bigcap_{n\in\bZ} f^n(U)$. To avoid trivialities we will always assume
that $h_{{\rm top}}(f|\Lambda)>0$, where $h_{{\rm top}}$ denotes the topological entropy of the map. This rules out the case that $\Lambda$ is
only a periodic orbit. Given $x\in \Lambda$ and $v\in T_xM$, we
define the \emph{Lyapunov exponent} of $v$ at $x$ (with respect to
$f$) by
\begin{equation}\label{deflya}
\lambda(x,v)\eqdef\limsup_{n\to\infty}\frac{1}{n}\log\,\lVert df^n(x)(v)\rVert
\end{equation}
with the convention that $\log 0=-\infty$.
For each $x\in \Lambda$ there exist a positive integer $s(x)\le \dim M$,
real numbers $\chi_1(x)< \cdots < \chi_{s(x)}(x)$, and linear spaces
$\{0\}=E^{0}_x\subset \cdots \subset E^{s(x)}_x=T_xM$ such that for
$i=1$, $\ldots$, $s(x)$ we have
\[
E^{i}_x=\{v\in T_xM\colon \lambda(x,v)\le \chi_i(x)\},
\]
and $\lambda(x,v)=\chi_i(x)$ whenever $v\in E^{i}_x\setminus 
E^{i-1}_x$.

We will count the values of the Lyapunov exponents
$\chi_i(x)$ with their multiplicities, i.e. we consider the numbers
\[
\lambda_1(x)\le\cdots\le\lambda_{\dim M}(x),
\]
where  $\lambda_j(x)=\chi_i(x)$ for each $j\in\{\dim E^{i-1}_x+1,\cdots,\dim
E^i_x\}$.

Let $\cM$ denote the set of all Borel $f$-invariant probability measures on
$\Lambda$ endowed with the weak$*$ topology. This makes $\cM$
a compact convex space. Moreover, denote by $\cM_{\rm E}\subset \cM$ the subset
of ergodic measures.
By the Oseledets theorem, given $\mu\in \cM$ the set of Lyapunov regular points (i.e. the set of points where the limit superior in~\eqref{deflya} is actually a limit) has full measure and
$\lambda_i(\cdot)$ is $\mu$-measurable. We denote by
\begin{equation}\label{deflyame}
\lambda_i(\mu)\eqdef\int\lambda_i(x) d\mu(x)
\end{equation}
the Lyapunov exponents of the measure $\mu$.
Note that if $\mu\in \cM_{\rm E}$ then
$\lambda_i(.)$ is constant $\mu$-a.e., and therefore, the corresponding value
coincides with $\lambda_i(\mu)$.
For $\mu\in \cM$
set
\[
\chi(\mu)\eqdef\min_{i=1,\ldots,\dim M} \lambda_i(\mu) =\lambda_1(\mu).
\]
Furthermore, we define $\cM^+=\{\mu\in \cM\colon \chi(\mu)>0\}$ and
$\cM^+
_{\rm E}= \cM^+
\cap \cM_{\rm E}$.

 We denote
by $\Fix(f)$ the set of fixed points of $f$. Moreover, we
denote by $\Per(f)=\bigcup_{n\in\bN} \Fix(f^n)$ the set of periodic points
of $f$. For $x\in \Fix(f^n)$ we have that $\lambda_i(x)=\frac{1}{n}\log
|\delta_i|$, where $\delta_i$ are the eigenvalues of $df^n(x)$.
We say a periodic point $x$ is \emph{expanding} if $\lambda_1(x)>0$. Let $\EFix(f^n)$ denote the fixed points of $f^n$ which are expanding. Hence,
$\EPer(f)=\bigcup_{n\in \bN} \EFix(f^n)$ is the set of all expanding periodic points.

Given $\ell\in \bN$ and $\alpha>0$ we define
\begin{equation}\label{ne}
\Lambda_{\alpha,\ell}=\{x\in M\colon\lVert (df^k(x))^{-1}\rVert^{-1} \geq \frac{1}{\ell}e^{k\alpha} \text{ and all }k\in\bN\}.
\end{equation}
Here  $\lVert A^{-1}\rVert^{-1}$ is minimum norm of a linear transformation $A$.
Moreover, we say that a compact forward invariant set $K\subset M$ is  \emph{uniformly expanding} if there exist constants  $\ell\in\bN$ and $\alpha>0$ such that $K\subset \Lambda_{\alpha,\ell}$.
For convenience we sometimes also refer to relative compact forward invariant sets contained in some $\Lambda_{\alpha,\ell}$ as uniformly expanding sets.
We denote by $\chi(K)$ the largest $\alpha>0$ such that $K\subset \Lambda_{\alpha,\ell}$ for some $\ell\in\bN$.
It is easy to see that if $x\in \EFix(f^n)$ then there exists  $\ell=\ell(x)\in\bN$ such that for all
integers $k\ge 0$ and $0\le i\le n-1$ we have
\begin{equation}\label{eqhi}
\frac{1}{\ell}e^{k\lambda_{1}(x)} \leq \lVert (df^k(f^i(x)))^{-1}  \rVert^{-1}.
\end{equation}
For $\ell\in\bN, 0<\alpha$ and $n\in\bN$ we define
\begin{multline}\label{eqexpand}
\EFix(f^n,\alpha,\ell) = \{x\in \EFix(f^n)\colon \lVert(df^{-k}(f^i(x)))^{-1}\rVert^{-1}
\ge \frac{1}{\ell} e^{k\alpha}\colon\\
 \text{ for all } k\in \bN \text{ and } 0\le i\le n-1 \}.
\end{multline}
It follows that, if  $\alpha\geq \alpha'$, $\ell\le \ell'$, then
\begin{equation}\label{ni}
\EFix(f^n,\alpha,\ell) \subset \EFix(f^n,\alpha',\ell')
\end{equation}
and
\begin{equation}
\EPer(f)=\bigcup_{\alpha>0}\bigcup_{\ell\ge 1}\bigcup_{n=1}^\infty
\EFix(f^n,\alpha,\ell).
\end{equation}
We will frequently need the following construction of uniformly expanding sets. 
Let $\alpha>0$ and $\ell\in \bN$ such that $\EFix(f^n,\alpha,\ell)\not=\emptyset$ for some $n\in \bN$. We define
\begin{equation}\label{wagner}
K = K_{\alpha,\ell} 
= \overline{\bigcup_{n=1}^\infty \EFix(f^n,\alpha,\ell)} .
\end{equation}
A simple continuity argument shows that $K$ is a uniformly expanding set.
Furthermore, for every $n\in\bN$  we have 
\begin{equation}\label{eqsi}
  \EFix(f^n)\cap K = \EFix(f^n,\alpha,\ell).
\end{equation}

Next,  we introduce a version of topological pressure which is entirely determined by the values of the potential on the expanding periodic points.

Let us first recall the classical topological pressure. Let
$(\Lambda,d)$ be a compact metric space and let $f\colon
\Lambda\to \Lambda$ be a continuous map. For $n \in {\mathbb N}$ we
define a new metric $ d_n $ on $ \Lambda$ by $
d_n(x,y)=\max_{k=0,\ldots ,n-1} d(f^k(x),f^k(y))$. A set of points
$\{ x_i\colon i\in I \}\subset \Lambda$ is called
\emph{$(n,\varepsilon)$-separated} (with respect to $f$) if
$d_n(x_i,x_j)> \varepsilon$ holds for all $x_i,x_j$ with $x_i \ne
x_j$. Fix for all $\varepsilon>0$ and all $n\in\bN$ a maximal
(with respect to the inclusion) $(n,\varepsilon)$-separated set
$F_n(\epsilon)$. The \emph{topological pressure} (with respect to $f|\Lambda$) is a mapping
$ P_{\rm top}(f|\Lambda,\cdot)\colon C(\Lambda,\bR)\to \bR$  defined by
\begin{equation}\label{defdru}
  P_{\rm top}(f|\Lambda,\varphi) = \lim_{\varepsilon \to 0}
            \limsup_{n\to \infty}
            \frac{1}{n} \log \left(\sum_{x\in F_n(\epsilon)}
            \exp S_n\varphi(x) \right),
\end{equation}
where
\begin{equation}\label{eqsn}
S_n\varphi(x)=\sum_{k=0}^{n-1}\varphi(f^k(x)).
\end{equation}
Recall that
$h_{\rm top}(f|\Lambda)=P_{\rm top}(f|\Lambda,0)$. We simply write $P_{\rm
  top}(\varphi)$ and $h_{\rm top}(f)$ if there is no confusion about $f$ and
$\Lambda$.
Note that the definition of $P_{\rm top}(\varphi)$ does not depend on the choice of the
sets $F_n(\epsilon)$ (see~\cite{Wal:81}).
The topological pressure satisfies the
well-known variational principle, namely,
\begin{equation}\label{eqvarpri}
P_{\rm top}(\varphi)= 
\sup_{\mu\in \cM} \left(h_\mu(f)+\int_\Lambda \varphi\,d\mu\right).
\end{equation}
Furthermore, the supremum in~\eqref{eqvarpri} can be replaced by
the supremum taken only over all $\mu\in\cM_{\rm E}$. 

We now introduce a \emph{pressure} which is entirely defined  by the
values of $\varphi$ on the expanding periodic points. Let $\varphi\in
C(\Lambda, \bR)$ and let $0<\alpha$, $\ell\in\bN$.
Define
\begin{equation}
 Q_{\rm EP}(\varphi,\alpha,\ell,n) =
\sum_{x\in \EFix(f^n,\alpha,\ell)} \exp S_n\varphi(x)
\end{equation}
if $\EFix(f^n,\alpha,\ell)\ne\emptyset$ and
\begin{equation}
Q_{\rm EP}(\varphi,\alpha,\ell,n)=
\exp\left(n\min_{x\in\Lambda} \varphi(x)\right)
\end{equation}
otherwise. Furthermore, we define
\begin{equation*}
P_{\rm EP}(\varphi,\alpha,\ell) = \limsup_{n\to\infty}
                 \frac{1}{n}\log Q_{\rm EP}(\varphi,\alpha,\ell,n).
\end{equation*}
It follows from the  definition that if
$\EFix(f^n,\alpha,\ell)\not=\emptyset$ for some $n\in\bN$
then this is already true for infinitely many $n\in\bN$. Therefore,
in the case when $\EFix(f^n,\alpha,\ell)\not=\emptyset$ for
some $n\in\bN$ then $P_{\rm EP}(\varphi,\alpha,\ell)$ is
entirely determined by the values of $\varphi$ on
$\bigcup_{n\in\bN}\EFix(f^n,\alpha,\ell)$.

We will need the following classical result (see for example~\cite{Ru}).

\begin{proposition}\label{ha}
 Let $f\colon M\to M$ be a $C^{1+\epsilon}$ map and let $K\subset M$ be
 a compact uniformly expanding  set of $f$. Then for every $\varphi\in C(K,\bR)$ we have,
 \begin{equation}\label{eqdruck}
 \limsup_{n\to\infty}\frac{1}{n}\log\left(\sum_{x\in\Fix(f^n)\cap K}\exp
 S_n\varphi(x)\right) \le P_{\rm top}(f|K,\varphi).
 \end{equation}
 Furthermore, if $f|K$ is topologically mixing then we have
 equality in~\eqref{eqdruck}, and the limsup is in fact a limit.
\end{proposition}

\section{Topological Pressure and expanding periodic points }\label{sec:3}

In this section we prove a theorem which is closely related to a result by Bowen~\cite{Bow:71}  concerning the relationship
between the topological pressure and the value of the potential on the expanding periodic orbits. A related  result was proved in~\cite{GelWol:08} in the context of non-uniformly hyperbolic diffeomorphisms.

In the following, $f:M\to M$ is a $C^{1+\varepsilon}$ map on a smooth Riemannian manifold $M$ and $\Lambda\subset M$ is a compact locally maximal invariant set.
Katok established the remarkable result that every hyperbolic invariant measure of a diffeomorphism can be  in a particular sense approximated by  uniformly hyperbolic horseshoes (see \cite[Theorem~S.5.9]{KatHas:95} for the precise statement). We will need a version of this result
in the context of non-uniformly expanding maps (see~\cite{Buz:} and \cite{Gel:}).  

\begin{proposition}\label{buthm}
	Let  $\mu\in\cM_E^+$ with $h_\mu(f)>0$ and let $\varphi_1$,$\ldots$, $\varphi_m\in C(\Lambda,\bR)$. 
	Then for every $\varepsilon>0$ there exists a compact uniformly expanding set $\Lambda_\varepsilon\subset \Lambda$ such that $f|\Lambda_\varepsilon$ is topologically conjugate to a mixing subshift of finite type  and satisfies 
	\begin{itemize}
	\item[(a)] $\chi(\Lambda_\varepsilon)>\lambda_1(\mu)-\varepsilon$,\vspace{0.1cm}
	\item [(b)] $\displaystyle 
	h_{\rm top}(f|\Lambda_\varepsilon)> h_\mu(f)-\varepsilon$, and\vspace{0.1cm}
	\item [(c)] for all $i=1,...,m$ and all $x\in\Lambda_\varepsilon$,
	\begin{equation}\label{appr}
		\lim_{n\to\infty}
		\left\lvert \frac{1}{n}\sum_{k=0}^{n-1}\varphi_i(f^k(x))- 
			\int_\Lambda\varphi_i \,d\mu\right\rvert 
		<\varepsilon.
	\end{equation}
	\end{itemize}
\end{proposition}

We now introduce a natural class of potentials. For $\varphi\in
C(\Lambda,\bR)$ set
\begin{equation}\label{defal}
\alpha(\varphi)\eqdef
P_{\rm top}(\varphi)-\sup_{\nu\in\cM}\int_\Lambda\varphi
\, d\nu.
\end{equation}
We say that a potential $\varphi$ belongs to $C^f(\Lambda,\bR)$ if
\begin{enumerate}
\item [(a)] $\alpha(\varphi)>0$; \item [(b)] there exist $0<\delta(\varphi)<\alpha(\varphi)$ and a
sequence $(\mu_n)_{n\in\bN}\subset \cM_{\rm E}^+$ 
   such that
  $\chi(\mu_n)>\delta(\varphi)$ for every $n\in\bN$ 
and
  $h_{\mu_n}(f)+\int_\Lambda \varphi\,d\mu_n \to P_{\rm top}(\varphi)$ as
  $n\to\infty$.
\end{enumerate}

\begin{remark}{\rm 
(i) Note that $\alpha(\varphi)\geq 0$, and $\alpha(\varphi)>0$ if and only if
$\varphi$ has no equilibrium state with zero entropy. \\
(ii) We note that if $\varphi$ has an  equilibrium
measure $\mu_\varphi\in \cM_E^+$ then we can simply choose  the constant sequence
$\mu_n=\mu_\varphi$ in (b).\\
(iii) Given $\rho>0$, $\varphi\in C(\Lambda,\bR)$ is said to have
\emph{$\rho$-low variation} if $\max_{x\in\Lambda}\varphi(x)<P_{\rm
  top}(\varphi)-\rho \,h_{\rm top}(f)$. In this case we clearly have
$\alpha(\varphi)\ge \rho \,h_{\rm top}(f)$, and hence property (a) holds since, by assumption,  $h_{\rm top}(f)>0$.
}\end{remark}

We will need the following auxiliary result. 

\begin{proposition}\label{li}
  Let $\varphi\in C(\Lambda,\bR)$. Then for all $\mu\in \cM_{\rm E}^+$ and for all
  $0<\alpha<\chi(\mu)$ we have 
  \begin{equation}\label{eqws}
  h_\mu(f)+\int_\Lambda\varphi \, d\mu \le
  \lim_{\ell\to \infty} P_{\rm EP}(\varphi,\alpha,\ell).
  \end{equation}
\end{proposition}

\begin{proof}
Let $\mu\in \cM_{\rm E}^+$ and $0<\alpha<\chi(\mu)$. 
Using the fact that on a compact uniformly expanding set every continuous potential has at least one ergodic equilibrium state and applying Proposition~\ref{buthm} implies the existence of a
sequence  $(\mu_n)_{n\in\bN}$ of measures  in $\cM_E^+$ supported on compact
uniformly expanding sets $K_n\subset \Lambda$ such that 
\begin{equation}\label{holl}
  h_\mu(f)+\int_\Lambda\varphi \,d\mu 
\le \liminf_{n\to\infty} h_{\mu_n}(f)+\int_\Lambda \varphi\, d\mu_n=
  \liminf_{n\to\infty}P_{\rm top}(f|K_n,\varphi),
\end{equation}
$\chi(\mu_n)>\alpha$ for all $n\in\bN$ 
and $\mu_n\to\mu$ with respect to the weak$\ast$ topology.
Moreover, for each $n\in\bN$ there exist $m=m(n)\in \bN$ 
such that $f^m| K_n$ is conjugate to a full one-sided
shift. For every 
$0<\varepsilon<\chi(\mu)-\alpha$ we pick $n=n(\varepsilon)\in \bN$
such that
\begin{equation}\label{kuh}
h_\mu(f)+\int_\Lambda\varphi\, d\mu -\varepsilon
\le P_{\rm top}(f|K_n,\varphi) .
\end{equation}
Moreover, by construction of the sets $K_n$, there exists a positive integer $\ell=\ell(n)$  such
that for every periodic point $x\in K_n$ and every $k\in\bN$ we have
\begin{equation}\label{eq15}
  \frac{1}{\ell}e^{k\alpha} <
  \lVert (df^k(x))^{-1}\rVert^{-1}.
\end{equation}
Therefore,
\begin{equation}\label{kir}
  \EFix(f^k)\cap K_n \subset \EFix(f^k,\alpha,\ell)
\end{equation}
holds for all $k\in\bN$. 
Let $m\in\bN$ such that $f^m|K_n$ is topologically conjugate to the full
one-sided shift.
Since $mP_{\rm top}(f|K_n,\varphi) = P_{\rm top}(f^m|K_n,S_m\varphi)$
(see~\cite[Theorem 9.8]{Wal:81}), we may conclude that
\begin{equation}
  h_\mu(f)+\int_\Lambda\varphi\, d\mu -\varepsilon
  \le \frac{1}{m}P_{\rm top}(f^m|K_n,S_m\varphi).
\end{equation}
It now follows from Proposition \ref{ha} and an elementary calculation that
\begin{equation}\label{eqrep}
  h_\mu(f) +\int_\Lambda \varphi\, d\mu -\varepsilon
    \le \lim_{k\to\infty}\frac{1}{k}\log\left(\sum_{z\in\EFix(f^k)\cap K_n}
    \exp S_{k}\varphi(z)\right).
\end{equation}
Combining~\eqref{kir} and~\eqref{eqrep} yields
\begin{equation}
  h_\mu(f)+\int_\Lambda \varphi\, d\mu -\varepsilon
  \le \limsup_{k\to\infty}\frac{1}{k}\log\left(
  \sum_{z\in\EFix(f^k,\alpha,\ell)}\exp S_k\varphi(z)\right).
\end{equation}
Finally, since  the map $\ell\mapsto P_{\rm
  EP}(\varphi,\alpha,\ell)$ is non-decreasing (see \eqref{ni})
 the proof is complete.
\end{proof}
We are now in the situation to present the main result of this section.   

\begin{theorem}\label{Main1}
  Let $f:M\to M$ be $C^{1+\varepsilon}$ map, let $\Lambda\subset M$ be a compact locally maximal invariant set, and let $\varphi\in C^f(\Lambda,\bR)$.
Then for all $\alpha\in(0,\delta(\varphi)]$ we have
  \begin{equation}\label{GW1}
  P_{\rm top}(\varphi) =   
  \lim_{\ell\to\infty}\limsup_{n\to\infty}\frac{1}{n}\log\left(
  \sum_{x\in \EFix(f^n,\alpha,\ell)} \exp S_n\varphi(x)\right).
  \end{equation}
\end{theorem}

\begin{proof}
Let $0<\alpha$ and $\ell\in\bN$ such that $\EFix(f^n,\alpha,\ell)\ne \emptyset$ for some $n\in\bN$. 
We first observe that
\begin{equation}\label{le}
  P_{\rm EP}(\varphi,\alpha,\ell) \le
  \sup_\nu \left(h_\nu(f)+  \int_\Lambda \varphi\, d\nu \right)
  \leq P_{\rm top}(\varphi),
\end{equation}
where the supremum is taken over all $\nu\in \cM_{\rm E}^+$ with
$\alpha\le\chi(\nu)$.  
Indeed, the right hand side inequality in~\eqref{le} is a
consequence of the variational principle while the 
left hand side inequality in~\eqref{le} can be shown analogously as in the proof of \cite[Theorem 1]{GelWol:08} by using the uniformly expanding set
$K = K_{\alpha,\ell} $ defined  in \eqref{wagner}.
 To complete the proof we shall now show that
\begin{equation}\label{pressp}
  P_{\rm top}(\varphi)\le \lim_{\ell\to\infty}P_{\rm EP}(\varphi,\alpha,\ell).
\end{equation}
Let $0<\alpha\leq\delta(\varphi)$ and let $\varepsilon>0$.
It follows from the the definition of $\delta(\varphi)$ that there exists  $\mu\in\cM_{\rm E}^+$ with $\chi(\mu)>\delta(\varphi)\geq\alpha$ such that  
\begin{equation}\label{gu}
  P_{\rm top}(\varphi)-\varepsilon
  \leq h_{\mu}(f) + \int_\Lambda \varphi\, d\mu.
\end{equation}
Moreover, by  Proposition~\ref{li} we have
\begin{equation}\label{eqwer}
  h_{\mu}(f)+\int_\Lambda\varphi\, d\mu
  \le \lim_{\ell\to \infty}P_{\rm EP}(\varphi,\alpha,\ell).
\end{equation}
Since $\varepsilon$ can be chosen arbitrary small,~\eqref{gu}
and~\eqref{eqwer} imply~\eqref{pressp}. 
\end{proof}

\begin{remark}
 {\rm 
As mentioned before, Theorem~\ref{Main1} is a version of~\cite[Theorem~1]{GelWol:08} in the case of a non-uniformly expanding (and in particular not necessarily invertible) map $f$.
}
\end{remark}

\section{Asymptotic distribution of periodic points}\label{sec:4}
In this section we  consider a  $C^{1+\varepsilon}$ map  $f:M\to M$ and a fixed  compact local maximal invariant set $\Lambda\subset M$.
Our goal is to   generalize the concept of a Bowen measure to the non-uniformly expanding case.   
Let $\varphi\in C(\Lambda, \bR)$ be a fixed potential. For  $n\in \bN$ and $\alpha>0$ we define
\begin{equation}
\EFix(f^n,\alpha)=\{x\in \Fix(f^n): \alpha \leq \lambda_1(x)\}.
\end{equation}
Hence,
\begin{equation}
\EFix(f^n,\alpha)=\bigcup_{\ell\in \bN} \EFix(f^n,\alpha,\ell).
\end{equation}
Moreover, for $E_n\subset \EFix(f^n,\alpha)$ we define
\begin{equation}
A(E_n) =
 \sum_{x\in E_n}\exp{S_n\varphi(x)}
\end{equation}
and define  measures $\sigma_n(E_n)\in \cM$ by  
\begin{equation}\label{tilsig}
\sigma_n(E_n)=
\frac{1}{A(E_n)}
\sum_{x\in E_n}\exp\left(S_n\varphi(x)\right)\delta_x.
\end{equation}
We note that by compactness of $\Lambda$ all the sets $\EFix(f^n,\alpha)$ (and therefore in particular the sets $E_n$) are finite. 
We say that a measure $\mu\in \cM^+_E$  is \emph{well-approximated by expanding periodic points}
if there  exist $\alpha< \lambda_1(\mu)$, a  subsequence $(n_k)_{k\in \bN}$, and for all $k\in \bN$ a set $E_{n_k}\subset \EFix(f^{n_k},\alpha)$ such that
\begin{equation}\label{sigmu}
\sigma_{n_k}(E_{n_k})\stackrel{\rm weak*}{\longrightarrow}
	 \ \mu\ \text{\rm as}\  k\to\infty.
\end{equation}
Moreover, we call $\mu \in\cM^+_E$ a \emph{Bowen measure} for the potential $\varphi$ (with respect to $f|\Lambda$) if 
\[
 P_{\rm top}(\varphi)=h_\mu(f)+\int_\Lambda \varphi \,d\mu
\]
and if $\mu$ is well-approximated by expanding periodic points, where the weak$\ast$ convergence~\eqref{sigmu} holds for $E_{n_\ell}= \EFix(f^{n_\ell},\alpha,\ell)$ for some subsequence $(n_\ell)_{\ell\in \bN}$.
If  $\mu$ is a Bowen measure with respect to  the potential $\varphi\equiv 0$ then we call $\mu$ a Bowen measure for the entropy.

\subsection{Well--approximated and Bowen measures}

We now consider a general $C^{1+\varepsilon}$ map  $f:M\to M$ and a fixed compact $f$-invariant local maximal set $\Lambda\subset M$. We start by showing that hyperbolic measures are well-approximated by expanding periodic orbits.

\begin{theorem}\label{thwell}
 Let $f:M\to M$ be a $C^{1+\varepsilon}$ map, let $\Lambda\subset M$ be a locally maximal invariant set and let  $\mu\in\cM^+_E$ with $h_\mu(f)>0$. Then  $\mu$ is well-approximated by expanding periodic points.
\end{theorem}

\begin{proof}
 Let $(\varphi_k)_{k\in \bN}$ be a dense subset of the separable Banach space $C(\Lambda,\bR)$. For $\nu$, $\eta\in \cM$ we define
\[
 d(\nu,\eta)=\sum_{i=1}^\infty \frac{\left\lvert\int \varphi_i\,d\nu - \int \varphi_i\,d\eta\right\rvert}{2^i\lVert \varphi_i\rVert}.
\]
It is well-known that $d$ is a metric which induces the weak$\ast$ topology on $\cM$. Let $0<\alpha<\lambda_1(\mu)$.  Given $k\in \bN$ we chose $m\in \bN$ with
$2\sum_{i=m+1}^{\infty}2^{-i}< (1/3) 2^{-k}$ and $\varepsilon_k>0$ small enough such that 
\begin{equation}
 \sum_{i=1}^m \frac{\varepsilon_k}{2^i\lVert \varphi_i\rVert}<\frac{1}{3} 2^{-k}.
\end{equation}
Making $\varepsilon_k$ smaller if necessary we can assure that $\alpha<\lambda_1(\mu)-\varepsilon_k$ and $\varepsilon_k<2^{-k}$.
Let $\Lambda_k$ be the uniformly expanding set obtained by applying Proposition \ref{buthm} to $\mu$, $\varepsilon=\varepsilon_k$ and to the functions $\varphi_1$, $\ldots$, $\varphi_m$. In particular, $f|{\Lambda_k}$ is topologically conjugate to a mixing subshift of finite type, $\chi(\Lambda_k)>\alpha$, and
$h_{\rm top}(f|{\Lambda_k})>h_\mu(f)-\varepsilon_k$. Let $\nu_k$ denote the unique measure of maximal entropy of
$f|{\Lambda_k}$.  Applying the Birkhoff ergodic theorem to~\eqref{appr} and using the definition of $m$ and $\varepsilon_k$ we conclude that $d(\nu_k,\mu)<\frac{2}{3}2^{-k}$.  It is well-known that $\nu_k$ is a Bowen measure for $f|{\Lambda_k}$.
Thus, we can chose $n_k\in \bN$ such that 
\[
d(\sigma_{n_k}(\Fix(f^{n_k})\cap \Lambda_k),\nu_k)<\frac{1}{3} 2^{-k}.
\] 
Defining $E_{n_k}=\Fix(f^{n_k})\cap \Lambda_k$ we obtain that $d(\sigma_{n_k}(E_{n_k}),\mu)<2^{-k}$. Moreover, by
defining the $n_k$ successively we can assure that $n_k<n_{k+1}$ for all $k\in\bN$. This completes the proof.
\end{proof}

We will need the following result, which is typically shown by using the Misiurewicz argument when proving the variational principle (see, for example,~\cite[Section 2.4]{UrbPrz:}).

\begin{lemma}\label{lemKH}
 Let $f:M\to M$ be a $C^{1+\varepsilon}$ map, let $\Lambda\subset M$ be a locally maximal invariant set, and let $\varphi\in C(\Lambda, \bR)$.
 For $n\in \bN$ let $E_n$ be a sequence of maximal $(n,\varepsilon)$-separated sets in $\Lambda$, and define the measures
\[
\nu_n=\sigma_n(E_n),
\quad
\mu_n=\frac{1}{n}\sum_{i=1}^{n-1}f^i_*\nu_n.
\] 
Then there exists a weak$\ast$ accumulation point $\mu$ of the measures $(\mu_n)_{n\in \bN}$ such that
\begin{equation}\label{lemmis}
 \limsup_{n\to \infty}\frac{1}{n}\log\sum_{x\in E_n}\exp S_n\varphi(x) 
\leq h_\mu(f) +\int_\Lambda\varphi\,d\mu.
\end{equation}
\end{lemma}

The next result shows the Bowen measure property for a large class of non-uniformly  expanding  equilibrium states. 

\begin{theorem}\label{thmut}
 Let $f:M\to M$ be a $C^{1+\varepsilon}$ map and let $\Lambda\subset M$ be a locally maximal invariant set such that the entropy map $\nu\mapsto h_\nu(f)$ is upper semi-continuous on $\cM$. 
Assume that  $\mu\in\cM^+_E$ is the unique equilibrium state of a potential $\varphi\in C(\Lambda,\bR)$  and that $h_\mu(f)>0$. Then  $\mu$ is a Bowen measure with respect to $\varphi$.
\end{theorem}

\begin{proof}
We first notice that since $\mu$ is the unique equilibrium state of the potential $\varphi$, $\mu\in \cM_E^+$, and $h_\mu(f)>0$, it follows that $\varphi\in C^f(\Lambda,\bR)$.  Let $\delta(\varphi)$ be as in the definition of $C^f(\Lambda,\bR)$ and 
let $\alpha\in (0,\delta(\varphi)]$.  For $\ell\in \bN$ we define
\begin{equation}\label{eqzxc}
 p_\ell=\limsup_{n\to\infty}\frac{1}{n}\log\left(
  \sum_{x\in \EFix(f^n,\alpha,\ell)} \exp S_n\varphi(x)\right).
\end{equation}
Therefore,~\eqref{GW1} implies that
\begin{equation}\label{eqert}
 h_\mu(f)+\int_\Lambda \varphi d\mu =P_{\rm top}(\varphi)=\lim_{\ell\to\infty}p_\ell.
\end{equation}
Given  $\ell\in\bN$  let 
$K = K_{\alpha,\ell} $ be the uniformly expanding set defined  in~\eqref{wagner}. 
Therefore, for all $n\in \bN$  with $\EFix(f^n,\alpha,\ell)\ne\emptyset$ we have $\EFix(f^n)\cap K = \EFix(f^n,\alpha,\ell)$, see~\eqref{eqsi}.
It follows that $f|K$ is expansive with some expansivity constant $\varepsilon=\varepsilon_\ell>0$.
Note that $E_n\eqdef \EFix(f^n,\alpha,\ell)$ is an $(n,\epsilon)$ separated set.
Since $\cM$ is compact (also using~\eqref{eqzxc} and~\eqref{eqert}) there exist $\mu_\ell\in \cM$ and a subsequence
$(n_k)_{k\in\bN}$ such that $\lim_{k\to \infty} \sigma_{n_k}(E_{n_k})=\mu_\ell$ and 
\begin{equation}
 p_\ell=\lim_{k\to\infty}\frac{1}{n_k}\log\left(
  \sum_{x\in E_{n_k}} \exp S_{n_k}\varphi(x)\right).
\end{equation}
Hence, by~\eqref{lemmis},
\begin{equation}\label{eq123}
 p_\ell\leq h_{\mu_\ell}(f)+\int_\Lambda \varphi \,d\mu_{\ell}.
\end{equation}
We now pick $n_\ell\in \{n_k: k\in \bN\}$ such that 
\begin{equation}\label{ende}
 d(\sigma_{n_\ell}(E_{n_\ell}),\mu_\ell)<2^{-\ell}.
\end{equation}
Moreover, by
defining  $n_\ell$ successively we can assure that $n_\ell<n_{\ell+1}$ for all $\ell\in\bN$.
It follows from~\eqref{eqert},~\eqref{eq123}, and the variational principle that
\begin{equation}
 \lim_{\ell\to \infty} h_{\mu_\ell}(f)+\int_\Lambda\varphi \,d\mu_\ell = P_{\rm top}(\varphi).
\end{equation}
Let $\nu\in \cM$ be any weak$\ast$ accumulation point of the sequence of measures  $(\mu_\ell)_{\ell\in \bN}$. Then, 
the upper semi-continuity of the entropy map implies $h_{\nu}(f)+\int\varphi\, d\nu = P_{\rm top}(\varphi)$.
Using that $\mu$ is the unique equilibrium state of $\varphi$ we may conclude that $\nu=\mu$. In particular,
$\mu_\ell\to \mu$ as $\ell\to \infty$. Finally, \eqref{ende} yields $\lim_{\ell\to\infty}\sigma_{n_\ell}(E_{n_\ell})=\mu$, which completes the proof.
\end{proof}

\begin{remark}\label{rem:3}
{\rm 
 We note that in the proof of Theorem~\ref{thmut} we do not really need upper semi-continuity of the entropy map at every $\mu\in \cM$. Indeed, one can easily verify that Theorem~\ref{thmut} remains true if the entropy map  is upper semi-continuous at all measures $\nu\in \cM$ which can be approximated by a sequence of  measures $(\nu_n)_{n\in \bN}$ such that $h_{\nu_n}(f)+\int_\Lambda \varphi \,d\nu_n$ converges to $P_{\rm top}(\varphi)$. We will explore this further in the next section.
}\end{remark}                   

\subsection{Expansive maps}

In addition to the assumptions in Theorem~\ref{thmut} we now assume that $f$ is expansive on  $\Lambda$.
In this case our methods even provide a stronger approximation property by expanding periodic points.

\begin{theorem}\label{thmut2}
 Let $f:M\to M$ be a $C^{1+\varepsilon}$ map and let $\Lambda\subset M$ be a locally maximal invariant set such that $f|\Lambda$ is expansive. 
Assume that  $\mu\in\cM^+_E$ is the unique equilibrium state of a potential $\varphi\in C(\Lambda,\bR)$ and that $h_\mu(f)>0$.
 Then  for all $0<\alpha\leq\delta(\varphi)$ we have 
\begin{equation}\label{eqbowga}
 \lim_{n\to\infty}\sigma_{n}(\EFix(f^{n},\alpha))=\mu.
\end{equation}
\end{theorem}

\begin{proof}
We use the fact that expansivity implies that the entropy map is upper semi-continuous (see~\cite{Wal:81}).
Therefore, the proof  can be done along the lines of the proof of Theorem~\ref{thmut}. There are only two differences.
First we do not have to consider a fixed $\ell$ when looking at periodic points in $\EFix(f^{n},\alpha)$
since $\EFix(f^{n},\alpha)$ is already a $(n,\varepsilon)$-separated set if $\varepsilon$ is an expansivity constant for $f|\Lambda$. Second, if $\nu\in \cM$ is any accumulation point of the sequence of measures $(\sigma_{n}(\EFix(f^{n},\alpha)))_{n\in \bN}$, then by a similar argument as in the proof of Theorem \ref{thmut} we obtain $h_\nu(f)+\int_\Lambda \varphi \,d\nu= P_{\rm top}(\varphi)$, which implies $\nu=\mu$, and we can conclude that   $\sigma_{n}(\EFix(f^{n},\alpha))$ converges to the measure $\mu$ as $n\to\infty$.
\end{proof}

\begin{example}{\rm 
	Let $f:\oc\to\oc$ be a rational map on the Riemann sphere with degree $d\geq 2$. Let $J\subset \oc$ denote the Julia set of $f$. We refer to the overview article~\cite{U} for more details concerning the ergodic theory of $f|J$. Suppose $f$ is parabolic, that is, the Julia set $J$ does not contain critical points but contains at least one parabolic point. In this case $f|J$ is expansive but not uniformly hyperbolic.  It follows that with $\varphi_t(z)= -t\log\,\lvert f'(z)\rvert$ and $t\in[0,\dim_{\rm H} J)$ the assumptions of Theorem \ref{thmut2} are satisfied for the potential $\varphi_t$. Thus~\eqref{eqbowga} holds for the corresponding equilibrium state $\mu_t$. Another interesting class of rational maps are the so-called  topological Collet-Eckmann maps~\cite{PRS1}. For these maps every  equilibrium
state is hyperbolic (in fact even every invariant measure is hyperbolic) and we obtain again a large class of equilibrium states which are Bowen measures.}
\end{example}

\begin{remark}{\rm
 The proofs of Theorems \ref{thmut} and \ref{thmut2} crucially rely on 
     the fact that the topological pressure can be computed by evaluating the potential on the hyperbolic periodic orbits (see Theorem \ref{Main1}). We recently proved an analogous statement in the case of  non-uniformly hyperbolic   diffeomorphisms and a rather general class of potentials  (see~\cite{GelWol:08} for details).
Therefore, by using similar methods as in this paper, we are able to prove versions of Theorems \ref{thmut} and \ref{thmut2} in the case of non-uniformly hyperbolic diffeomorphisms.\\
     Bedford et al. showed in \cite{BedLyuSmi:93}  that the unique measure of maximal entropy of a complex H\'enon map is a Bowen measure. This is a version of our results, in particular, since in this case the map is smooth and hence the entropy map is upper semi-continuous~\cite{New:89}.
}\end{remark}

\subsection{A class of non-uniformly expanding maps}

Following Remark~\ref{rem:3}, we now study a particular class of maps which have sufficiently strong expansion properties, though they are not necessarily expansive nor uniformly expanding. In particular, this class contains certain non-uniformly expanding maps which were studied in~\cite{Oli:03}.

We first collect some preliminary results.
We say a measure $\mu\in\cM_{\rm E}$ is \emph{$f$-expanding} with
exponent $\beta>0$, if for $\mu$-almost every $x$ we have
	\[
	\liminf_{n\to\infty}\frac{1}{n}
	\sum_{k=0}^{n-1}\log \,\lVert
          (df(f^k(x)))^{-1}\rVert^{-1} 
        \ge \beta.
	\]
We denote by $\cN(f,\beta)$ the set of all $f$-expanding measures with exponent $\beta$.
It is not hard to see that  $\mu\in\cM^+_{\rm E}$ (i.e. $\chi(\mu)>0$) if and only if for any $0<\beta<\chi(\mu)$ there exists $n\in\bN$ such that $\mu$ is $f^n$-expanding with exponent $\beta$. We note that any other  measure $\nu\in\cM^+_{\rm E}$ satisfying $\chi(\nu)>\beta$ is also $f^k$-expanding with exponent $\beta$ for some  $k\in\bN$, but in general $n$ and $k$ may be different.
Using this property, we will now show that   $f$ is expansive on a rather large  set of points with an expansivity constant depending only on $f^n$ and $\beta$. 

\begin{lemma}\label{lem:2}
	Assume that $f$ is a $C^1$  local diffeomorphism in a neighborhood of $\Lambda$.
	Let $n\in\bN$ and $\beta>0$. Then there exists $\delta=\delta(n,\beta)>0$ such that for any $\mu\in\cN(f^n,\beta)$ and any $0<\varepsilon<\delta$  we have that
	\[
	\left \{y\in M \colon d(f^k(x),f^k(y))\le \varepsilon\text{ for all }k\in \bN_0
	\right\} =\{x\}
	\]
for $\mu$-almost every $x\in\Lambda$.
\end{lemma}

\begin{proof}
  	We only sketch the proof and refer to~\cite{Oli:03} for further details.
	The claimed property is a consequence of the fact that for every $x$ and every so-called hyperbolic time $n_k\in\bN$ for $x$ with exponent $\alpha$ the size of the local unstable manifolds at  $f^{n_k}(x)$ are uniformly bounded away from zero. 
	Here a number $m\in\bN$ is called a hyperbolic time for $x$ with exponent $\alpha$ if $\lVert (df^k(f^m(x)))^{-1}\rVert^{-1}\ge e^{k\alpha}$ for every $1\le k\le m$.	It is an immediate consequence of the Pliss lemma~\cite{Pli:72} that for every point $x$ with strictly positive Lyapunov exponents greater than $\alpha$ there exist infinitely many hyperbolic times with exponent $\alpha$.
\end{proof}

\begin{theorem}\label{theorem:lisa}
  Let $f\colon M\to M$ be a $C^{1+\varepsilon}$ local diffeomorphism,  let $\Lambda\subset M$ be a locally maximal invariant set, and let $\varphi\in C(\Lambda,\bR)$. Assume that there exist $\rho>0$, $\beta>0$, and $n\in\bN$ such that $\alpha(\varphi)>\rho \,h_{\rm top}(f)$ and every measure $\nu\in\cM_{\rm E}$ satisfying $h_\nu(f)>\rho \,h_{\rm top}(f)$ is $f^n$-expanding with exponent $\beta$. Then there exists an equilibrium state of the potential $\varphi$. Moreover, if $f|\Lambda$ is topologically transitive then $\mu$ is
the unique equilibrium state of $\varphi$ in which case $\mu$ is a Bowen measure. 
\end{theorem}

\begin{proof}
	Let $\varphi\in C(\Lambda,\bR)$ and let $\rho, \beta$ and $n$ as in the theorem. By the variational principle there exists a sequence $(\mu_\ell)_{\ell\in\bN}$   in $\cM_E$ with
\begin{equation}
 P_{\rm top}(\varphi)=\lim_{\ell\to\infty} h_{\mu_\ell}(f)+\int_\Lambda \varphi \,d\mu_{\ell}.
\end{equation}
Let $\epsilon>0$ such that $\alpha(\varphi)-\epsilon >\rho \,h_{\rm top}(f)$.
Then for $\ell$ large enough (say $\ell\ge N$ for some $N\in \bN$) we obtain
\begin{equation}\label{htiek}
h_{\mu_\ell}(f)\ge P_{\rm top}(\varphi)-\int_\Lambda \varphi \,d\mu_{\ell}-\varepsilon
\ge \alpha(\varphi)-\varepsilon > \rho\,h_{\rm top}(f).
\end{equation}
We conclude that each of the measures $\mu_{\ell}, \ell \geq N$ is $f^n$-expanding with exponent $\beta$.
Let $\mu\in \cM$ be any weak$\ast$ accumulation point of the sequence of measures  $(\mu_\ell)_{\ell\in \bN}$.  
Let $\delta=\delta(n,\beta)$ be the number given by Lemma~\ref{lem:2} and let $\cP$ be some finite partition of $\Lambda$ into measurable sets of diameter less than $\delta$, and assume that $\mu(\partial\cP)=0$.
Note that by~\eqref{htiek} and by~\cite[Corollary 5.3.]{Oli:03} the partition $\cP$ of $\Lambda$ into measurable sets with diameter less than $\delta$ is generating for every $\mu_\ell$, $\ell\ge N$. Therefore, $h_{\mu_\ell}(f)=h_{\mu_\ell}(f,\cP)$.

We now establish the upper semi-continuity of the entropy map on the
particular set of measures $\{\mu_\ell\colon \ell\in\bN\}$. 
We have
\[
h_\mu(f) \ge h_\mu(f,\cP)
=\lim_{m\to\infty}\frac{1}{m}
H_\mu\left(\bigvee_{k=0}^{m-1}f^{-k}\cP\right).
\]
From the weak$\ast$ convergence $\mu_\ell\to\mu$ and from
$\mu(\partial\cP)=0$ it follows that for every fixed $m\in\bN$ 
\[
\frac{1}{m} H_\mu\left(\bigvee_{k=0}^{m-1}f^{-k}\cP\right)
= \lim_{\ell\to\infty} \frac{1}{m}
H_{\mu_\ell}\left(\bigvee_{k=0}^{m-1}f^{-k}\cP\right).
\]
Note that for every fixed $\ell\geq N$ the value of
$\frac{1}{m}H_{\mu_\ell}(\bigvee_{k=0}^{m-1}f^{-k}\cP)$ 
monotonically decreases to $h_{\mu_\ell}(f,\cP)$ as $m$ increases. 
Moreover, 
$h_{\mu_\ell}(f,\cP)= h_{\mu_\ell}(f)$, and we conclude $h_\mu(f)\ge \lim_{\ell\to\infty} h_{\mu_\ell}(f)$. From here we obtain that $\mu$ is an equilibrium state of $\varphi$. Finally, if $f|\Lambda$ is topologically transitive then by the same arguments as in  \cite{Oli:03} we obtain that $\mu$ is the unique equilibrium state
and therefore Theorem \ref{thmut} implies that $\mu$ is a Bowen measure.
 \end{proof}

\begin{example}{\rm 
	Oliveira introduced in \cite{Oli:03} a class of maps satisfying the hypothesis of Theorem~\ref{theorem:lisa}. Examples of these maps can be obtained from sufficiently small perturbations of  expanding maps. 
    Consider  a $C^{1+\varepsilon}$ local diffeomorphism $f$ on a compact  Riemannian manifold  $M$. Let $\cP\eqdef\{B_1$, $\ldots$, $B_p$, $B_{p+1}$, $\ldots$, $B_{p+q}\}$ be a covering of $M$ by connected closed sets such that $f|B_i$ is injective for each $1\le i\le  p+q$, that $\cP$ is transitive, and that there exist numbers $\sigma>q$, $\beta>0$, $\delta_0>0$, and $\gamma>0$ satisfying\\[-0.4cm]
\begin{itemize}
\item [] $\lvert \det df(x)\rvert\ge \sigma$ for every $x\in M$ (volume expansion everywhere),\\[-0.3cm]
\item []  $\lVert df(x)^{-1}\rVert\le 1+\delta_0$ for every $x\in M$ (not too strong contraction everywhere), \\[-0.3cm]
\item [] $\lVert df(x)^{-1}\rVert\le (1+\delta_1)^{-1}$ 
  if $x\in B_1\cup\ldots\cup B_p$ (some uniform expansion).
\end{itemize}
Assume further that  there exists a set $W\subset B_{p+1}\cup\ldots\cup B_{p+q}$ which contains the set $V$ given by 
  \[
  V\eqdef\{x\in M\colon \lVert df(x)^{-1}\rVert>(1+\delta_1)^{-1}\},
  \]
such that 
  \[
  \sup_V\log\,\lvert\det df\rvert < 
  \inf_{M\setminus W}\log\,\lvert\det df\rvert,
  \quad
  \sup_V\log\,\lvert df\rvert - \inf_V\log\,\lvert df\rvert < \beta.
  \]
Assuming that $\delta_0$, $\beta$, and $\rho>0$ are sufficiently small, it now follows from~\cite{Oli:03}  that for any continuous potential $\varphi$ which has $\rho$-low variation there exists an  equilibrium state $\mu_\varphi$. Moreover, $\mu_\varphi$ is unique and hence  Theorem \ref{thmut} implies that $\mu_\varphi$ is a Bowen measure.  
}\end{example}

\section{Large deviations}\label{sec:LD}

In this section we consider a $C^{1+\varepsilon}$ map  $f:M\to M$ and a fixed compact $f$-invariant local maximal set $\Lambda\subset M$.
It is a consequence of Proposition~\ref{buthm} that  for any $\mu\in \cM^+_E$ with $h_\nu(f)>0$, any $\varphi\in C(\Lambda,\bR)$, and any given $\varepsilon>0$ there exists an expanding periodic orbit, say of period $n$, satisfying 
\[
\left\lvert \frac{1}{n}\sum_{k=0}^{n-1}\varphi(f^k(x))- \int_\Lambda\varphi \,d\mu\right\rvert 
		<\varepsilon.
\]
We now are interested in the periodic orbits that \emph{do not} have such an approximation
property and  will derive  results on the distribution of such orbits. This can be interpreted as a level-2 large deviation analysis. 

In order to formalize our large deviation results, we first introduce a rate
function. Closely related approaches can be found in~\cite{Yur:05,Pol:95}. 
Let $\cP(\Lambda)$ be the space of all Borel probability measures on $\Lambda$ endowed with the topology of weak$\ast$ convergence.
Given  $\varphi\in C(\Lambda,\bR)$, we define $Q_\varphi\colon C(\Lambda,\bR)\to
\bR$ by  
 \[
	Q _\varphi(\psi)= P_{\rm top}(\varphi+\psi)-P_{\rm top}(\varphi).
\] 
 We note that
$Q _\varphi$ is a continuous and convex functional. 
Within the standard framework of the theory of conjugating functions (see for
example~\cite{AubEke:84}), $Q_\varphi$ can be characterized by  
\[
Q_\varphi(\psi) = 
\sup_{\nu\in\cP} \left(\int_\Lambda\psi \,d\nu - I_\varphi(\nu)\right),
\]
 where $I_\varphi$ is the convex conjugate  of $Q_\varphi$  defined by 
\begin{equation}\label{Idef}
	I_\varphi(\mu) =
	\sup_{\psi\in C(\Lambda,\bR)}\left( \int_\Lambda\psi \,d\mu - Q_\varphi(\psi)
        \right) 
\end{equation}
for all $\mu\in\cM$ and $I_\varphi(\mu) =\infty$ for any other signed measure $\mu$. Since $I_\varphi\colon \cP \to \bR$ is a pointwise supremum of 
continuous and affine functions, it is a  lower semi-continuous  and convex
functional. 

Given $\nu\in \cP$, we call
\begin{equation}\label{defunas}
\widehat h_\nu(f) \eqdef 
-\sup_{\psi\in C(\Lambda,\bR)}
\left(\int_\Lambda\psi \,d\nu - P_{\rm top}(\psi)\right)
\end{equation}
the \emph{generalized entropy} of $f$ with respect to $\nu$. We call $\mu\in\cP$ a \emph{generalized equilibrium state} for $\varphi\in C(\Lambda,\bR)$ if $P_{\rm top}(\varphi)=\widehat h_\mu(f)+\int_\Lambda\varphi\,d\mu$. It follows from
the  definition that 
\begin{equation}\label{engen}
h_\nu(f)\le\widehat h_\nu(f)
\end{equation} 
for every $\nu\in \cM$. 
The choice of the above denotation is justified by the ``dual variational
principle" which says that the entropy map $\mu\mapsto h_\mu(f)$ is upper
semi-continuous at $\nu\in\cM$ if and only if $	h_\nu(f) = \widehat h_\nu(f)$
(see~\cite{Wal:81}). 
We clearly have 
 $I_\varphi(\mu)\geq 0$
for every $\mu\in\cM$. Moreover,
\[\begin{split}
P_{\rm top}(\varphi)& - 
\left(\widehat h_\mu(f)+\int_\Lambda\varphi \,d\mu \right) \\
&= P_{\rm top}(\varphi) - \int_\Lambda\varphi \,d\mu 
+ \sup_{\psi\in C(\Lambda,\bR)}\left(\int_\Lambda(\psi+\varphi) \,d\mu - P_{\rm top}(\psi+\varphi)\right)\\
&=  \sup_{\psi\in C(\Lambda,\bR)}\left(\int_\Lambda \psi \,d\mu 
- P_{\rm top}(\psi+\varphi)+P_{\rm top}(\varphi)\right) = I_\varphi(\mu)
\end{split}\]
and hence that
$I_\varphi(\mu)=0$ if and only if $\mu$ is a generalized
equilibrium state for $\varphi$. By this property, on $\cM$ the functional $I_\varphi$ can be interpreted as the
``distance" of the measure $\mu$ from the set of all generalized equilibrium states of $\varphi$. Moreover,
if the entropy map is upper semi-continuous at $\mu\in\cM$ then 
\begin{equation}
 I_\varphi(\mu) = 
P_{\rm top}(\varphi) - \left(h_\mu(f)+\int_\Lambda\varphi \,d\mu\right).
 \end{equation}
We notice that in the case that $\nu_\varphi$ is an equilibrium Gibbs measure of a potential $\varphi$, $I_{\varphi}(\mu)$ is the relative entropy of $\mu\in\cM$ given $\nu_\varphi$, and
$I_{\varphi}(\mu)=0$ if and only if $\mu=\nu_\varphi$.

First we derive a more general large deviation result in the context of continuous maps on compact metric spaces.
We refer to the work of Pollicott~\cite{Pol:95} for a related result in the context of uniformly hyperbolic flows.
For a given $x\in\Fix(f^n)$ we use the notation $\omega_x\eqdef \frac{1}{n}\left(\delta_x+\ldots +\delta_{f^{n-1}(x)}\right)$.

\begin{theorem}\label{the:Pol}
  	Let $f\colon X\to X$ be a continuous map on a compact metric space $X$
        with $h_{\rm top}(f)<\infty$. Let
        $Z\subset X$ be compact  such that $f|Z$ is expansive and let $\varphi\in C(X,\bR)$. Then for any
        compact set $\cK\subset\cM$ of Borel $f$-invariant probability measures on $X$ 
        we have 
	\begin{equation}\label{kinggo}
	\limsup_{n\to\infty}
	\frac{1}{n}\log\sum\limits_{x\in\Fix(f^n)\cap Z,\omega_x\in \cK}
	\exp S_{n}\varphi(x)
        \le P_{\rm top}(\varphi)- \inf_{\nu\in\cK} I_\varphi(\nu).
        \end{equation}
\end{theorem}

\begin{proof}
Let $\delta$ be an expansivity constant  for $f|Z $. Consider
  $x,y\in Z$ with $f^n(x)=x$ and $f^n(y)=y$  satisfying
  $d(f^k(x),f^k(y))\le \delta$ for every $0\le k\le n-1$. It follows
  that $d(f^k(x),f^k(y))\le \delta$ for all $k\in\bN$, and hence
  $x=y$. Therefore, the set $\Fix(f^n)\cap  Z$ is
  $(n,\delta)$-separated, and thus there exists a maximal $(n,\delta)$-separated set $F_n(\delta)$ (with respect to $f$ on $X$) with $\card \left(\Fix(f^n)\cap  Z\right) \le \card \left(F_n(\delta)\cap Z\right)<\infty$. This implies
  \begin{equation}\label{bekha}
  \limsup_{n\to\infty}\frac{1}{n}\log
    \sum\limits_{x\in\Fix(f^n)\cap Z}\exp{S_{n}\varphi(x)}\\ 
    \le P_{\rm top}(\varphi).  
  \end{equation}

  Let $\rho\eqdef\inf_{\nu\in\cK}I_\varphi(\nu)$. Fix 
  $\varepsilon>0$. Given $\nu\in\cK$, by definition~\eqref{Idef} we have 
  \begin{equation}
  \rho\le\sup_{\psi\in C(X,\bR)}\left(\int_X\psi \,d\nu-Q_\varphi(\psi)\right).
  \end{equation}
  Hence there exists $\psi=\psi(\nu,\varepsilon)\in C(X,\bR)$ such that $\rho -\varepsilon <
  \int_X \psi \,d\nu -Q_\varphi(\psi)$.
  Thus, we obtain that
  \begin{equation}
  \cK \subset \bigcup_{\psi\in  C(X,\bR)}
  \left\{\nu\in\cM\colon \int_X\psi \,d\nu -
    Q_{\varphi}(\psi)>\rho-\varepsilon\right\}. 
  \end{equation}
  By compactness of $\cK$ there exists a finite cover
  $\cU_1$, $\ldots$, $\cU_N$ of $\cK$ determined by functions $\psi_1$,
  $\ldots$, $\psi_N\in C(X,\bR)$ by 
  \begin{equation}
  \cU_i\eqdef	\left\{\nu\in\cM\colon 
    \int_X\psi_i \,d\nu - Q_\varphi(\psi_i)-\rho +\varepsilon
    >0 \right\}.
  \end{equation}
  We now can conclude that
  \[
\begin{split}
&  \sum\limits_{x\in\Fix(f^n)\cap Z,\omega_x\in \cK}
    \exp {S_n\varphi(x)}
    \\
&\phantom{\sum} \le 
    \sum_{i=1}^N 
    \sum\limits_{x\in\Fix(f^n)\cap Z,\omega_x\in \cU_i}\exp {S_n\varphi(x)}\\
  & \phantom{\sum} \le 
    \sum_{i=1}^N 
    \sum\limits_{x\in\Fix(f^n)\cap Z ,\omega_x\in \cU_i}\exp{S_n\varphi(x)}\cdot
    \exp \left[n\left(\int_X \psi_i \,d\omega_x - Q_\varphi(\psi_i)-\rho+\varepsilon\right)\right]\\
    & \phantom{\sum} \le 
    \sum_{i=1}^N \exp [-n(Q_\varphi(\psi_i)+\rho-\varepsilon)]	         
    \sum\limits_{x\in\Fix(f^n)\cap Z }
    \exp {S_n(\varphi+\psi_i)(x)}.
  \end{split}\]
  From \eqref{bekha} we derive that
  \begin{equation}
\begin{split}
    \limsup_{n\to\infty}&\frac{1}{n}\log
    \sum\limits_{x\in\Fix(f^n)\cap Z,\omega_x\in
      \cK}\exp{S_{n}\varphi(x)}\\ 
    &\le \max_{1\le i\le N}
    \left\{ -Q_\varphi(\psi_i)-\rho+\varepsilon+P_{\rm top}(\varphi+\psi_i)\right\}\\  
    &= P_{\rm top}(\varphi)-\rho+\varepsilon		      	.
  \end{split}\end{equation}
  Since $\varepsilon>0$ was arbitrary, the statement is
  proved.
\end{proof}

As an application of Theorem \ref{the:Pol} we finally derive a large deviation result. 

\begin{theorem}\label{theorem:ld}
	Let $f\colon M\to M$ be a $C^{1+\varepsilon}$ map and let $\Lambda\subset M$ be a locally maximal invariant set. Let $\mu\in\cM^+_{\rm E}$ with $h_\mu(f)>0$ and $\varphi\in C(\Lambda,\bR)$.  Then for all  $\alpha>0$ and  $\delta>0$ we have
	\begin{multline}\label{kingo}
	\lim_{\ell\to\infty}\limsup_{n\to\infty}\frac{1}{n}\log
	{\rm card}\left\{x\in\EFix(f^n,\alpha,\ell) \colon 
	\left\lvert \frac{1}{n}S_n\varphi(x) - \int_\Lambda\varphi \,d\mu\right\rvert \ge \delta \right\}\\
	        \le 
		\sup\left\{\widehat h_\nu(f)\colon 
	        \nu\in\cM, \left\lvert \int_\Lambda\varphi \,d\nu-\int_\Lambda\varphi \,d\mu\right\rvert
	        \ge \delta \right\}.
	\end{multline}
\end{theorem}

\begin{proof}
  Let $\cK=\left\{\nu\in\cM\colon \left\lvert \int_\Lambda\varphi \,d\nu-\int_\Lambda\varphi \,d\mu\right\rvert\ge \delta\right\}$.
  Note that, by definition we have 
\[
h_{\rm top}(f) - \inf_{\nu\in\cK}I_0(\nu)
= h_{\rm top}(f) - \sup_{\psi\in C(\Lambda,\bR)}
\left( \int_\Lambda\psi\,d\nu - P_{\rm top}(\psi)+h_{\rm top}(f) \right) = 
\widehat h_\nu(f).
\]
 The proof follows now from an application of Theorem \ref{the:Pol} with $Z=K_{\alpha,\ell}$.
\end{proof}


\end{document}